\documentclass[12pt]{amsart}

\usepackage{hyperref}
\usepackage{amsfonts}
\usepackage{amsmath}
\usepackage{amssymb}
\usepackage{amsthm}
\usepackage{enumerate} 
\usepackage{hyperref}
\hypersetup{
	colorlinks=true,
	citecolor=blue
}
\usepackage[margin=1.2in]{geometry}
\usepackage{mathrsfs} 
\usepackage{graphicx} 
\usepackage{tikz-cd}
\usepackage{wasysym}
\usepackage{xspace}

\usepackage{xcolor}

\usepackage[margin=1.2in]{geometry}

 \newtheorem{thm}{Theorem}[section]
 
 \newtheorem{lem}[thm]{Lemma}
 \newtheorem{prop}[thm]{Proposition}
 \theoremstyle{definition}
 \newtheorem{defn}[thm]{Definition}
 \theoremstyle{remark}
 \newtheorem{rem}[thm]{Remark}
 
 \numberwithin{equation}{section}
 
 \newcommand\N{\mathbb{N}}
 \newcommand\Z{\mathbb{Z}}
 \newcommand\R{\mathbb{R}}
 \newcommand\C{\mathbb{C}}
 
 \renewcommand\S{\mathcal{S}}
 
\newcommand\open{U}
\newcommand\M{\mathcal{M}}

 \newcommand\ev[2]{\langle#1,#2\rangle}

\DeclareMathOperator\id{id}

\DeclareMathOperator\condS{S}
\DeclareMathOperator\condA{A}
\DeclareMathOperator\DN{DN}
\DeclareMathOperator\uDN{\underline{DN}}

\begin{document}

\title[The vector-valued Stieltjes moment problem with general exponents]
 {The vector-valued Stieltjes moment problem with general exponents}

\author[A. Debrouwere]{Andreas Debrouwere}

\address{%
Department of Mathematics and Data Science \\ Vrije Universiteit Brussel, Belgium\\ Pleinlaan 2 \\ 1050 Brussels \\ Belgium}

\email{andreas.debrouwere@vub.be}

\author[L. Neyt]{Lenny Neyt}
\address{Universit\"{a}t Trier \\ FB IV Mathematik \\ D-54286 Trier \\ Germany}
\email{lenny.neyt@ugent.be}
\thanks{L. Neyt gratefully acknowledges support by the Alexander von Humboldt Foundation and by FWO-Vlaanderen through the postdoctoral grant 12ZG921N}



\begin{abstract}
We characterize the sequences of complex numbers  $(z_{n})_{n \in \N}$ and the  locally complete $(DF)$-spaces $E$ such that for each  $(e_{n})_{n \in \N} \in E^\N$  there exists an $E$-valued function $\mathbf{f}$ on $(0,\infty)$ (satisfying a mild regularity condition) such that
$$
\int_{0}^{\infty} t^{z_{n}} \mathbf{f}(t) dt = e_{n}, \qquad \forall n \in \N,
$$
where the integral should be understood as a Pettis integral.  Moreover, in this case,  we show that there always exists a solution $\mathbf{f}$ that is smooth on $(0,\infty)$ and satisfies certain optimal growth bounds near $0$ and $\infty$. 
The scalar-valued case $(E = \C)$ was treated by Dur\'an [Math. Nachr. \textbf{158} (1992), 175--194]. Our work is based upon his result. 
\end{abstract}

\maketitle

\section{Introduction and main results}
Moment-type problems are a classical topic in analysis that go back to the pioneering work of Stieltjes \cite{S-RecherchFracCont} from 1894. Boas \cite{B-StieltjesMomProbFuncBoundedVar} and P\'olya \cite{P-IndeterminThProbMom} in 1939  independently solved the following (unrestricted) Stieltjes moment problem: For each sequence $(a_{n})_{n \in \N} \in \C^{\N}$ there exists a function $F$ of bounded variation on $(0,\infty)$ such that
	\[ \int_{0}^{\infty} t^{n} dF(t) = a_{n} , \qquad \forall n \in \N. \]
In 1989 Dur\'an  \cite{D-StieltjesMomProbRapidDecrFunc} greatly improved this result by showing that for each sequence $(a_{n})_{n \in \N} \in \C^{\N}$ the infinite system of linear equations 
	\[ \int_{0}^{\infty} t^{n} f(t)dt = a_{n} , \qquad \forall n \in \N, \]
admits a solution $ f \in \S(0, \infty)$ (= the space of rapidly decreasing smooth functions with
support in $[0, \infty)$). The original proof of Dur\'an was constructive. His result also follows from a short non-constructive argument via Eidelheit's theorem \cite{E-ThSystLinGleichung} (see also \cite[Theorem 26.27]{M-V-IntroFuncAnal}). We refer to \cite{C-C-K-StieltjesMomProbSolGSSp, D-SolStieltjesMomProbGSSp, E-V-GenStieltjesMomProbRapidDecrSmoothFunc, L-S-LinContOpStieltjesMomProbGSSp} for recent works related to the Stieltjes moment problem.

In \cite{D-StieltjesMomentProbComplexExp} Dur\'an considered the following Stieltjes moment problem  with general exponents: For which sequences of complex numbers $(z_{n})_{n \in \N}$ does it hold that for every sequence $(a_{n})_{n \in \N} \in \C^{\N}$ there exists a measurable function $f$ on $(0,\infty)$ such that
	\[ \int_{0}^{\infty} t^{z_n}f(t) dt = a_n , \qquad \forall n \in \N? \]
If this holds, does there always exist a  solution $f$ that is smooth on $(0,\infty)$ and satisfies certain optimal growth bounds near $0$ and $\infty$? Dur\'an gave a complete solution to this problem. We need to introduce two notions to state his result. Set $\R_{+} = (0,\infty)$. Let $\alpha, \beta \in \R \cup \{ -\infty, \infty\}$  with $\beta < \alpha$. We define $\S(\alpha, \beta)$ as the Fr\'echet space consisting of all $f \in C^{\infty}(\R_{+})$ such that for all $\gamma \in (\beta, \alpha)$ and $n \in \N$ it holds that
	\[ \| f \|_{\gamma,n} :=  \max_{0 \leq m \leq n}\sup_{t > 0} t^{\gamma + m + 1} |f^{(m)}(t)|  < \infty . \]
Let $(z_{n})_{n \in \N}$ be  a sequence of complex numbers and set $\alpha = \sup_{n \in \N} \Re e \, z_{n}$ and $ \beta =\inf_{n \in \N} \Re e \, z_{n}$. The sequence $(z_{n})_{n \in \N}$   is said to satisfy the condition $(\condS)$ if  $z_{n} \neq z_{m}$ for all $n \neq m$, and if one of the following conditions holds:
			\begin{itemize}
				\item $\alpha$ is the unique  accumulation point of $(\Re e \,  z_{n})_{n \in \N}$ and $\Re e \,  z_{n} \neq \alpha $ for all $n \in \N$.  
				\item $\beta$ is the unique  accumulation point of $(\Re e \,  z_{n})_{n \in \N}$ and $\Re e \,  z_{n} \neq \beta$ for all $n \in \N$.  
				\item $(\Re e \,  z_{n})_{n \in \N}$ has exactly two accumulation points, which are $\alpha$ and $\beta$, and $\Re e \,  z_{n} \neq \alpha, \beta$ for all $n \in \N$.
			\end{itemize}

\noindent  Dur\'an's solution to the Stieltjes moment problem  with general exponents now reads as follows:
	\begin{thm}[{\cite[Theorem 1.3]{D-StieltjesMomentProbComplexExp}}]
		\label{t:ClassicStieltjesMomentProblem}
		Let $(z_{n})_{n \in \N}$ be a sequence of complex numbers and set $\alpha = \sup_{n \in \N} \Re e \, z_{n}$ and $\beta = \inf_{n \in \N} \Re e \, z_{n}$. The following statements are equivalent:
			\begin{itemize}
				\item[$(i)$] $(z_{n})_{n \in \N}$ satisfies $(\condS)$.
				\item[$(ii)$] For every $(a_{n})_{n \in \N} \in \C^\N$ there exists a function $f: \R_{+} \to \C$ such that $t^{z_{n}} f(t) \in L^1(\R_+)$, $n \in \N$,   and
					\begin{equation}
	 					\int_{0}^{\infty} t^{z_{n}} f(t) dt = a_n, \qquad \forall n \in \N.
						\label{mpge}
					\end{equation}
				\item[$(iii)$] For every $(a_{n})_{n \in \N} \in \C^\N$ there exists  $f \in \S(\alpha, \beta)$ that satisfies \eqref{mpge}.
			\end{itemize}
	\end{thm}
	
A natural problem is to study whether Theorem \ref{t:ClassicStieltjesMomentProblem} may be extended to functions with values in a given locally convex space $E$.  Dur\'an showed that this is the case if $E$ is a Banach space  \cite[Theorem 3.2]{D-StieltjesMomentProbComplexExp}. More generally, this holds true if $E$ is a Fr\'echet space, as follows from the general theory of topological tensor products  (see Section \ref{sec:Eidelheit} for details). 
In this article, we study the above problem for  $E$ belonging to the class of locally convex Hausdorff spaces that are locally complete and have a fundamental sequence of bounded sets (the prime example being locally complete $(DF)$-spaces). This turns out to be much more delicate than the Fr\'echet case: Whether Theorem \ref{t:ClassicStieltjesMomentProblem} may be extended to $E$-valued functions depends crucially on the linear topological structure of the space $E$.  Let us mention that interpolation problems for (real) analytic functions with values in a sequentially complete $(DF)$-space have been studied in \cite{B-D-V-InterpolVVRealAnalFunc, B-LinTopStructClosedIdealsFAlg}. These works were one of the main motivations of the present article.

We now state and discuss our main result. We need some preparation. Let $E$ be a lcHs = (locally convex Hausdorff space). For $\alpha, \beta \in \R \cup \{ -\infty, \infty\}$  with $\beta < \alpha$ we define $\S(\alpha, \beta;E)$ as the space consisting of all $\mathbf{f} \in C^{\infty}(\R_{+};E)$ such that for all $\gamma \in (\beta, \alpha)$, $n \in \N$, and continuous seminorms $p$ on $E$ it holds that
			\[ \sup_{t > 0}  t^{\gamma + n + 1} p(\mathbf{f}^{(n)}(t))  < \infty . \]
All vector-valued integrals in this article should be understood as  Pettis integrals. A function $\mathbf{f} : \R_{+} \to E$ is said to be \emph{locally Pettis integrable} if the $E$-valued Pettis integral
$\int_K \mathbf{f}(t) dt$ exists for all compact subsets $K$ of $\R_{+}$. We are ready to state the main result of this article; its proof will be given in Section  \ref{sec:ProofMainThm}.

	\begin{thm}
		\label{t:main}
		Let $(z_{n})_{n \in \N}$ be a sequence of complex numbers and set $\alpha = \sup_{n \in \N} \Re e \, z_{n}$ and $\beta = \inf_{n \in \N} \Re e \, z_{n}$. Let $E$ be a non-zero locally complete lcHs with a fundamental sequence of bounded sets $(B_{N})_{N \in \N}$. The following statements are equivalent:
			\begin{itemize}
				\item[$(i)$] $(z_{n})_{n \in \N}$ satisfies $(\condS)$  and $E$ satisfies $(\condA)$, i.e.,
				\begin{gather*} 
			\exists N \in \N ~ \forall M \geq N ~ \forall \nu > 0 ~ \exists K \geq M, C > 0 ~ \forall r > 0 : \\
				B_{M} \subseteq r B_{K} + C r^{-\nu} B_{N} .
			\end{gather*}
				\item[$(ii)$]  For every $(e_{n})_{n \in \N} \in E^\N$ there exists a  locally Pettis integrable function $\mathbf{f}: \R_{+} \to E$ such that  
					\begin{equation}
						\label{eq:VVMomentsSol}
					 	\int_{0}^{\infty} t^{z_{n}} \mathbf{f}(t)  dt = e_n , \qquad \forall n \in \N.
					\end{equation}
				\item[$(iii)$]   For every $(e_{n})_{n \in \N} \in E^\N$ there exists $f \in S(\alpha, \beta;E)$ such that \eqref{eq:VVMomentsSol} holds.
			\end{itemize}
	\end{thm}
	
\noindent The linear topological invariant $(\condA)$ was introduced by Vogt  \cite{V-VektDistRandHolomFunk} and plays a pivotal role in the study of various vector-valued problems in analysis, see \cite{B-LinTopStructClosedIdealsFAlg} for interpolation problems for analytic functions and \cite{B-D-ParamDepSolDiffEqSpDistSplitSES, Vogt1983-2} for the surjectivity of PDO on spaces of smooth functions and distributions.  The condition $(\condA)$ is closely related to the well-known property $(\DN)$ for Fr\'echet spaces \cite{M-V-IntroFuncAnal}, e.g., a reflexive $(DF)$-space satisfies $(\condA)$ if and only if its strong dual satisfies $(\DN)$. Hence, duals of power series spaces of infinite type (in particular, the space of tempered distributions $\mathcal{S}'(\R^d)$) satisfy $(\condA)$, whereas duals of power series spaces of finite type (in particular, the space of holomorphic germs on a compact set $K$ in $\C^d$ that is the closure of a  bounded Reinhardt domain \cite[Theorem 5.5]{THD}) do not.

One of the main motivations for studying vector-valued problems is the question of solving equations depending on a parameter (the problem of parameter dependence). In our setting, this question reads as follows: Let $(z_{n})_{n \in \N}$ be a sequence of complex numbers satisfying $(\condS)$. Set $\alpha = \sup_{n \in \N} \Re e \, z_{n}$ and $\beta = \inf_{n \in \N} \Re e \, z_{n}$.  Let $(a_{\lambda})_{\lambda} = (a_{n,\lambda})_{n \in \N, \lambda}$ be a family in $\C^\N$ depending ``nicely" on the parameter $\lambda$ (e.g. smoothly, holomorphically, in a weighted discrete way, etc.). Does there exist a family $(f_{\lambda})_{\lambda} \subseteq S(\alpha, \beta)$ depending on $\lambda$ in the same fashion  as $(a_{\lambda})_{\lambda}$ such that
					$$ 
						\int_{0}^{\infty} t^{z_{n}} f_\lambda(t)  dt = a_{n,\lambda} , \qquad \forall n \in \N, \lambda?
					$$
This question is equivalent to $(iii)$ of Theorem \ref{t:main}  with $E$  an appropriately chosen function space corresponding to the type of parameter dependence under consideration. Hence, Theorem \ref{t:main} and the fact that $(iii)$ of this result always holds if $E$ is a Fr\'echet space (cf.\ Section \ref{sec:Eidelheit})  lead to various concrete instances of the problem of parameter dependence for the Stieltjes moment problem with general exponents. The following result is an illustration of this; we will show it in Section  \ref{sec:ProofMainThm} by using Theorem \ref{t:main}.

	\begin{thm}
		\label{t:MainResultSeqSp}
		Let $(z_{n})_{n \in \N}$ be a sequence of complex numbers satisfying $(\condS)$. Set $\alpha = \sup_{n \in \N} \Re e \, z_{n}$ and $\beta = \inf_{n \in \N} \Re e \, z_{n}$. 
		Let $\Lambda$ be some index set and let $(\omega_j)_{j \in \N}$ be a sequence of functions on $\Lambda$ such that  $0  < \omega_{j+1}(\lambda) \leq \omega_{j}(\lambda)$ for all $j \in \N$ and $\lambda \in \Lambda$.
		The following statements are equivalent:
			\begin{itemize}
							\item[$(i)$] $(\omega_j)_{j \in \N}$ satisfies		
							\begin{equation}
						\label{eq:AParamDep}
						\exists j \in \N ~ \forall k \geq j ~ \exists l \geq k, C > 0 ~ \forall \lambda \in \Lambda : \,
						\omega_j(\lambda) \omega_l(\lambda) \leq C \omega_k(\lambda)^{2} .
					\end{equation}	
					
					\item[$(ii)$] For every sequence $(c_{n, \lambda})\in \C^{\N \times \Lambda}$ satisfying
					\begin{equation}
						\label{eq:SeqParamDepCond} 
						\forall n \in \N ~ \exists j \in \N ~: ~\sup_{\lambda \in \Lambda} |c_{n, \lambda}|  \omega_j(\lambda) < \infty 
					\end{equation}
				there exist  measurable functions $f_{\lambda}: \R_+ \to \C$, $\lambda \in \Lambda$,  such that
				$$
				\forall n \in \N ~ \exists j \in \N ~:~  \sup_{\lambda \in \Lambda} \int_{0}^{\infty} |t^{z_{n}} f_{\lambda}(t)| dt  \, \omega_j(\lambda)  < \infty 
				$$
				and 
					\begin{equation}
						\label{seq-interpol}  \int_{0}^{\infty} t^{z_{n}} f_{\lambda}(t) dt = c_{n, \lambda} , \qquad \forall n \in \N, \lambda \in \Lambda . 
						\end{equation}
				\item[$(iii)$] For every sequence $(c_{n, \lambda})\in \C^{\N \times \Lambda}$ satisfying \eqref{eq:SeqParamDepCond} 
					there exist  $f_{\lambda} \in \S(\alpha, \beta)$, $\lambda \in \Lambda$, such that
					\begin{equation}
						\label{charvv}  
					\forall \gamma \in (\beta, \alpha), n \in \N ~ \exists j \in \N ~:~  \sup_{\lambda \in \Lambda} \| f
					_\lambda \|_{\gamma,n} \omega_j(\lambda) < \infty
				\end{equation}
					and \eqref{seq-interpol} holds.
					\end{itemize}
	\end{thm}

This article is organized as follows. In the preliminary Section \ref{sec:Preliminaries} we recall several notions and results about locally convex spaces that will be used later on. Locally Pettis integrable functions with values in a locally complete lcHs are investigated in Section \ref{sec:WeakInt}. The results from this section will be used to show the equivalence $(ii) \Leftrightarrow (iii)$ in Theorem \ref{t:main}. Next, in Section \ref{sec:Eidelheit}, we explain a result from our recent article \cite{D-N-ExtLBSpSurjTensMap} about general vector-valued Eidelheit-type problems. We will apply this result to show the equivalence $(i) \Leftrightarrow (iii)$ in Theorem \ref{t:main}.  To this end, we establish various linear topological properties of the spaces $\S(\alpha, \beta)$  and $\S(\alpha, \beta; E)$ in Section  \ref{sec:ProofMainThm}. Here, we also show Theorems \ref{t:main} and \ref{t:MainResultSeqSp}.

\section{Preliminaries}
\label{sec:Preliminaries}
Let $E$ be a lcHs. 
We denote by $E^{\prime}$ the dual of $E$. 
Given an absolutely convex bounded subset $B$ of $E$, we write $E_{B}$ for the subspace of $E$ spanned by $B$ endowed with the topology generated by the Minkowski functional of $E$. Since $E$ is Hausdorff, $E_B$ is normed. We call $B$ a \emph{Banach disk} if $E_{B}$ is a Banach space. If every bounded subset of $E$ is contained in some Banach disk, $E$ is said to be \emph{locally complete}. See \cite[Section 5.1]{BPC} and \cite[Chapter I, Section 2]{K-M} for more information on 
this condition. Each sequentially complete lcHs is locally complete \cite[Corollary 5.1.8]{BPC}.

Given two lcHs $E$ and $F$, we write $L(E, F)$ for the space of all continuous linear mappings $E \to F$.
The $\varepsilon$-product \cite{K-Ultradistributions3} of $E$ and $F$ is defined as
	\[ E \varepsilon F = L(E^{\prime}_{c}, F), \]
where the subscript $c$ indicates that we endow $E^{\prime}$ with the topology of uniform convergence on absolutely convex compact subsets of $E$.
The spaces $E \varepsilon F$ and $F \varepsilon E$ are canonically isomorphic via transposition \cite[p.~657]{K-Ultradistributions3}. 
If $E_{1}, E_{2}, F_{1}, F_{2}$ are lcHs, $T \in L(E_{1}, E_{2})$, and $S \in L(F_{1}, F_{2})$, we define the  mapping
	\[ T \varepsilon S : E_{1} \varepsilon F_{1} \to E_{2} \varepsilon F_{2} , \, R \mapsto S \circ R \circ T^{t} . \]

Let $E$ be a lcHs. For $\open \subseteq \R^{d}$ open we write $C^{\infty}(\open; E)$ for the space of $E$-valued smooth functions on $E$ \cite[Chapter 40]{T-TopVecSpDistKern}. Given $\mathbf{f} : \open \to E$ and $e' \in E'$, we define 
$$
\langle e', \mathbf{f} \rangle : U \to \C, \, x \mapsto \langle e', \mathbf{f}(x) \rangle.
$$
The following result is well-known (cf.\ \cite{BFJ}). 

	\begin{lem} 
		\label{l:WeakCm=>VVCm-1}
		Let $E$ be a locally complete lcHs and let $\open \subseteq \R^{d}$ be open. A function $\mathbf{f} : \open \rightarrow E$ belongs to $C^{\infty}(\open; E)$ if and only if $\langle e', \mathbf{f} \rangle \in C^\infty(U)$ for all $e' \in E'$. In such a case, 
		$$
		\langle e', \mathbf{f} \rangle^{(\alpha)} = \langle e', \mathbf{f}^{(\alpha)} \rangle, \qquad \forall e' \in E',  \alpha \in \N^d.
	$$
	\end{lem}

\begin{rem}
The following converse to Lemma \ref{l:WeakCm=>VVCm-1} holds \cite[Theorem 2.14]{K-M}: \emph{Let $E$ be a lcHs. Suppose that every $\mathbf{f} : \R \rightarrow E$ satisfying $\langle e', \mathbf{f} \rangle \in C^\infty(\R)$ for all $e' \in E'$ belongs to $C^\infty(\R;E)$. Then, $E$ is locally complete.} 
\end{rem}

\section{Locally Pettis integrable functions}
\label{sec:WeakInt}

Let  $E$ be a lcHs and let $U \subseteq \R^{d}$ be measurable. A function $\mathbf{f} : U \rightarrow E$ is called \emph{scalarly integrable} if $\ev{e'}{\mathbf{f}} \in L^{1}(U)$  for all $e' \in E^{\prime}$. A scalarly integrable function $\mathbf{f}:  U \rightarrow E$ is said to be \emph{Pettis integrable} if there exists $e \in E$ such that
	\[ \ev{e'}{e} = \int_U \ev{e'}{\mathbf{f}(x)} dx , \qquad \forall e' \in E^{\prime}. \]
In such a case, the element $e$ is unique (as $E$ is Hausdorff). We write  $e = \int_{U} \mathbf{f}(x) dx$ and call  it the  \emph{Pettis integral} of $\mathbf{f}$.  
		
		Let $E$ be a lcHs and let $\open \subseteq \R^d$ be open. A function $\mathbf{f}: \open \to E$ is said to be \emph{locally Pettis integrable} if for each $K \subseteq U$ compact the Pettis integral $\int_K \mathbf{f}(x) dx$
		exists (cf.\ the introduction). The following result is well-known, it follows from  \cite[Proposition 2]{BJM} and \cite[Theorem 3.27]{Rudin}.		
	\begin{lem}
		\label{c:WeakC1BornCont}
		Let $E$ be a locally complete lcHs and let $\open \subseteq \R^{d}$ be open.
		Let $\mathbf{f}: \open \to E$ be scalarly continuously differentiable, i.e., $\ev{e'}{\mathbf{f}} \in C^1(U)$ for all $e' \in E'$. Then, $\mathbf{f}: \open \to E$ is locally Pettis integrable.
	\end{lem}
	
	\begin{rem}
	$(i)$ The following converse to Lemma \ref{c:WeakC1BornCont} holds (cf.\ \cite[Theorem 2.14]{K-M}): \emph{Let $E$ be a lcHs. If every scalarly continuously differentiable function  $\mathbf{f}: \R \to E$ is locally Pettis integrable, then $E$ is locally complete.} \\
\noindent $(ii)$  A lcHs $E$ is said to satisfy the metric convex compactness property (for short, $(mcc)$) if the closed absolutely convex hull of every metrizable compact set in $E$ is again compact.  $(mcc)$ implies local completeness \cite[Theorem 5.1.11]{BPC}. This implication is strict (cf.\ \cite[Section 14.6, Problems 105-107]{Wilansky}). If $E$ satisfies $(mcc)$, then every continuous function  $\mathbf{f}: \open \to E$, $U \subseteq \R^d$ open, is  locally Pettis integrable \cite[Theorem 3.27]{Rudin}.  Conversely, if every continuous function $\mathbf{f}: \R \to E$ is  locally Pettis integrable, then $E$ satisfies $(mcc)$ \cite[Theorem 1]{HvW}. Hence, there are locally complete lcHs $E$ such that not every continuous function $\mathbf{f}: \R \to E$ is locally Pettis integrable.	
	\end{rem}
		
Let $U \subseteq \R^d$ be open and let  $v: \open \to (0,\infty)$ be continuous. We denote by $L^\infty_{v}(\open)$ the Banach space consisting of all (equivalence classes) of measurable functions $\varphi: \open \to \C$ such that 
	\[ \| \varphi \|_{L^\infty_{v}(\open)} := \underset{x \in \open}{\operatorname{ ess \, sup}} \frac{|\varphi(x)|}{v(x)} < \infty. \]
We write $C_{0,v}(\open)$ for the space consisting of all continuous functions $\varphi: U \to \C$ such that for each $\varepsilon >0$ there is a compact set $K \subseteq \open$ such that 
	\[ \sup_{x \in \open \backslash K } \frac{|\varphi(x)|}{v(x)} \leq \varepsilon . \]

The proof of the next result is inspired by the one of \cite[Proposition 14, p.~53--54]{S-TheorieDistValeurVect}. 

	\begin{prop}
		\label{p:weak-C0}
		Let $E$ be a locally complete lcHs, let $\open \subseteq \R^d$ be open, and let $v: \open \to (0,\infty)$ be continuous. Let  $\mathbf{f}: \open \to E$ be locally Pettis integrable and suppose that $v{\mathbf{f}}: U \to E$ is scalarly integrable.  Then, the Pettis integral 
		$
	\int_U \mathbf{f}(x)\varphi(x) dx
		$
		exists for all $\varphi \in C_{0,v}(\open)$.
	\end{prop}
	
		\begin{proof}
	Let $S \subseteq L^\infty_{v}(\open)$  be the linear span of the set consisting of the characteristic functions of all compact subsets of $U$. We endow $S$ with the norm $\| \, \cdot \, \|_{L^\infty_{v}(\open)}$. 
	 Note that the Pettis integral $\int_U \mathbf{f}(x)\varphi(x) dx$
	exists for each $\varphi \in S$. Consider the mapping
	$$
	T: S \to E, \, \varphi \mapsto \int_U \mathbf{f}(x)\varphi(x) dx.
	$$
	 Let $A$ be the unit ball in $S$. For each $e' \in E'$ it holds that 
			\[ 
				\sup_{\varphi \in A} | \ev{e'}{T(\varphi)} | 
				\leq  \sup_{\varphi \in A} \int_{\open} |\ev{e'}{\mathbf{f}(x)}| |\varphi(x)| dx 
				\leq \int_{\open}  |\ev{e'}{\mathbf{f}(x)}| v(x) dx. 
			\]
		Hence, $T(A)$ is weakly bounded and thus bounded in $E$. 
		Since $E$ is locally complete, this implies that there is a Banach disk $B \subseteq E$ such that $T: S \to E_B$ is continuous. Let $\overline{S}$ be the closure of $S$ in  $L^\infty_{v}(\open)$. There is a continuous linear mapping $\widetilde{T} : \overline{S} \to E$ such that $\widetilde{T}_{\mid S} = T$. Let $\varphi \in C_{0,v}(\open)$ be arbitrary. Since $C_{0,v}(\open) \subseteq \overline{S}$, there is a sequence $(\varphi_n)_{n \in \N} \subseteq S$ such that $\varphi_n \to \varphi$ in $ L^\infty_{v}(\open)$. Then, for  all $e' \in E'$
			\begin{align*} 
				\ev{e'}{\widetilde{T}(\varphi)} 
				&= \lim_{n \to \infty} \ev{e'}{T(\varphi_n)} 
				= \lim_{n \to \infty} \int_{\open} \ev{e'}{\mathbf{f}(x)} \varphi_n(x) dx \\
				&= \int_{\open} \ev{e'}{\mathbf{f}(x)} \varphi(x) dx, 
			\end{align*}
		where we used the Lebesgue dominated convergence theorem in the last equality. Therefore, the Pettis integral $\int_{\open} \mathbf{f}(x) \varphi(x) dx =\widetilde{T}(\varphi)$ exists.
\end{proof}

\section{Vector-valued Eidelheit sequences}
\label{sec:Eidelheit}

 Let $F$ be a Fr\'echet space. Eidelheit \cite{E-ThSystLinGleichung} (see also \cite[Theorem 26.27]{M-V-IntroFuncAnal}) characterized the sequences  $(x'_{n})_{n \in \N} \subseteq F'$ such that the infinite system of scalar-valued linear equations
	\[ \ev{x'_n}{x}  = a_n, \qquad \forall n \in \N, \]
has a solution $x \in F$ for each $(a_n)_{n \in \N} \in \C^{\N}$. In such a case, the sequence $(x'_{n})_{n \in \N}$ is called \emph{Eidelheit}. By definition, $(x'_{n})_{n \in \N}$ is Eidelheit if and only if the mapping
	\[ Q = Q_{(x'_{n})_{n \in \N}}: F \rightarrow \C^{\N} , \,  x \mapsto (\ev{x'_{n}}{x})_{n \in \N} \]
is surjective.  

\begin{rem}		\label{t:ClassicStieltjesMomentProblemEidelheit}
Let $\alpha, \beta \in \R \cup \{ -\infty, \infty\}$ with $\beta < \alpha$. For each $z \in \C$ with  $\Re e \, z \in (\beta, \alpha)$ the mapping
	\[ \M_{z} : \S(\alpha, \beta) \rightarrow \C, \quad \M_{z}(f) = \int_{0}^{\infty} t^{z} f(t) dt   \]
	is continuous. Hence, the equivalence $(i) \Leftrightarrow (iii)$ in Theorem \ref{t:ClassicStieltjesMomentProblem} may be reformulated as follows:
\emph{Let $(z_{n})_{n \in \N}$ be a sequence of complex numbers and set $\alpha = \sup_{n \in \N} \Re e \, z_{n}$ and $\beta = \inf_{n \in \N} \Re e \, z_{n}$. Then, $(\M_{z_{n}})_{n \in \N} \subseteq \S(\alpha, \beta)'$ is Eidelheit if and only if $(z_{n})_{n \in \N}$ satisfies $(\condS)$.}
\end{rem}
Let $F$ be a Fr\'echet space and let $(x'_{n})_{n \in \N} \subseteq F'$ be Eidelheit. Given a lcHs $E$, we define the \emph{associated $E$-valued sequence of  $(x'_{n})_{n \in \N}$} as 
	\[ (x'_{n} \varepsilon \id_E)_{n \in \N} \subseteq L( F \varepsilon E, E), \]
where we identified $\C \varepsilon E$ with $E$.
This sequence is called \emph{Eidelheit} if  the infinite system of $E$-valued  linear equations
	\[ \ev{x'_n \varepsilon \id_E}{x}  =e_n, \qquad \forall n \in \N, \]
has a solution $x \in F \varepsilon E$ for each $(e_n)_{n \in \N} \in E^\N$. Note that $(x'_{n} \varepsilon \id_E)_{n \in \N}$  is Eidelheit if and only if the mapping
	\[ Q \varepsilon \id_E: F \varepsilon E \rightarrow \C^\N \varepsilon E \cong E^\N \]
is surjective. 

We have the following natural problem: 
\emph{Let $F$ be a Fr\'echet space and let $(x'_{n})_{n \in \N} \subseteq F'$ be Eidelheit. Let $E$ be a lcHs. Find sufficient and necessary conditions on $E$ such that the associated $E$-valued sequence of  $(x'_{n})_{n \in \N}$ is also Eidelheit.}
	
If $F$ is a nuclear Fr\'echet space, then for each Fr\'echet space $E$ the associated $E$-valued sequence of any Eidelheit sequence $(x'_{n})_{n \in \N} \subseteq F^{\prime}$ is again Eidelheit: This follows from the equality $F \varepsilon E = F \widehat{\otimes}_\pi E$ \cite[Proposition 1.4]{K-Ultradistributions3} and the fact that the completed $\pi$-tensor product of two surjective continuous linear mappings between Fr\'echet spaces is again surjective \cite[Proposition 43.9]{T-TopVecSpDistKern}.

The above problem for lcHs $E$ with a fundamental sequence of bounded sets was studied by Vogt  \cite{V-TensProdFundDFRaumForsetz}. Recently, the authors reconsidered and extended his results  \cite{D-N-ExtLBSpSurjTensMap}.  The proof of  $(i) \Leftrightarrow (iii)$ in Theorem \ref{t:main} will be based on a result from this work. To formulate it, we need the following two linear topological invariants for Fr\'{e}chet spaces \cite{M-V-IntroFuncAnal}.

	\begin{defn}
		Let $F$ be a Fr\'{e}chet space with a fundamental increasing sequence of seminorms $(\|\cdot\|_{n})_{n \in \N}$.
			\begin{itemize}
				\item[$\bullet$] $F$ is said to satisfy $(\uDN)$ if
					\begin{gather*} 
						\exists n \in \N ~  \forall m \geq n ~ \exists k \geq m, \theta \in (0, 1), C > 0 ~\forall x \in F : \\
						\|x\|_{m} \leq C \|x\|_{n}^{\theta} \|x\|_{k}^{1 - \theta} . 
					\end{gather*}
					
				\item[$\bullet$] $F$ satisfies $(\Omega)$ if
					\begin{gather*}
						\forall n \in \N ~ \exists m \geq n ~ \forall k \geq m ~ \exists \theta \in (0, 1), C > 0 ~ \forall x' \in F^{\prime} : \\
						\|x'\|_{m}^{*} \leq C {\|x'\|_{n}^{*}}^{\theta} {\|x'\|_{k}^{*}}^{1 - \theta} ,
					\end{gather*}
				where $\|x'\|_{l}^{*} = \sup \{ |\ev{x'}{x}| \mid \|x\|_{l} \leq 1 \} \in [0,\infty]$.
			\end{itemize}
		Both conditions are independent of the choice of the fundamental increasing sequence of seminorms $(\|\cdot\|_{n})_{n \in \N}$ for $F$.
	\end{defn}

	 \begin{thm}[{\cite[Theorem 7.6]{D-N-ExtLBSpSurjTensMap}}]
		\label{t:Eid}
		Let $F$ be a nuclear Fr\'echet space satisfying $(\uDN)$. Let $(x'_n)_{n \in \N} \subseteq F'$ be an Eidelheit sequence such that 
			\[ \ker Q = \{ x \in F \, | \, \ev{x'_n}{x} = 0, \, \forall n \in \N \} \]
		satisfies $(\Omega)$. Let $E$ be a locally complete lcHs with a fundamental sequence of bounded sets. Then, the associated $E$-valued  sequence of $(x'_{n})_{n \in \N}$ is Eidelheit if and only if $E$ satisfies $(\condA)$.
	\end{thm}
	
\noindent Let $F$ be a Fr\'echet space and let $(x'_{n})_{n \in \N} \subseteq F'$ be Eidelheit. In general $\ker Q$ does not satisfy $(\Omega)$ even if $F$ does so \cite{V-Mit}. Braun \cite{B-LinTopStructClosedIdealsFAlg} showed that $\ker Q$ does inherit $(\Omega)$ from $F$ under certain additional assumptions on $F$ and  $(x'_n)_{n \in \N}$. We now state his result.

Let $F$ be a Fr\'{e}chet space and let $* : F \times F \rightarrow F$ be a bilinear mapping. $(F, *)$ is called an \emph{$m$-convex Fr\'{e}chet algebra}  if there exists an increasing fundamental sequence of seminorms $(\|\cdot\|_{n})_{n \in \N}$ on $F$ such that  for all $n \in \N$
	\[ \| x * y \|_{n} \leq \|x\|_{n} \|y\|_{n} , \qquad \forall x, y \in F. \]
	
	\begin{thm}[{\cite[Theorem 4.1]{B-LinTopStructClosedIdealsFAlg}}]
		\label{l:ClosedIdealHasOmega}
	Let $(F, *)$ be an $m$-convex Fr\'{e}chet algebra. Let $(x'_n)_{n \in \N} \subseteq F'$ be an Eidelheit sequence such that $x'_n: (F, *) \to \C$ is an algebra homomorphism for each $n \in \N$. If $F$ satisfies $(\Omega)$, then so does $
	\ker Q $.
	\end{thm}
	
\section{Proofs of the main results}
\label{sec:ProofMainThm}

This section is devoted to the proofs of Theorems \ref{t:main} and \ref{t:MainResultSeqSp}. We need several results in preparation.
Throughout this section we fix $\alpha, \beta \in \R \cup \{ -\infty, \infty\}$ with $\beta < \alpha$. We start by showing some linear topological properties of the space $\S(\alpha,\beta)$.

	\begin{prop}
		\label{c:S(a,b)Nuclear}		
		$\S(\alpha, \beta)$ is a nuclear Fr\'{e}chet space that satisfies $(\Omega)$ and $(\uDN)$. 	
	\end{prop}

	\begin{proof}
		We define $\mathcal{K}(\alpha, \beta)$ as the space consisting of all $F \in C^{\infty}(\R)$ such that 
			\[ \sup_{x \in \R} e^{\gamma x} |F^{(n)}(x)|  < \infty, \qquad \forall \gamma \in (\beta, \alpha), \, n \in \N, \]
		and endow it with its natural Fr\'echet space topology. Then,
			\[ \Phi: \S(\alpha,\beta) \to \mathcal{K}(\alpha, \beta), \quad \Phi(f)(x) = e^x f(e^x) \]
		is a topological isomorphism. We define $\Lambda(\alpha, \beta)$ as the space consisting of all $(c_{\nu,k})_{\nu \in \Z, k \in \N} \in \C^{\Z \times \N}$ such that 
			\[ \sup_{\nu \in \Z, k \in \N} |c_{\nu,k}|e^{\gamma \nu} k^n  < \infty, \qquad \forall \gamma \in (\beta, \alpha), \, n \in \N, \]
		and endow it with its natural Fr\'echet space topology. By using the exact same reasoning as in \cite[Theorem 3.1]{V-SeqSpRepSpTestFuncDist}, one can show that $\mathcal{K}(\alpha, \beta) \cong \Lambda(\alpha, \beta)$. Hence, it  suffices to show that $\Lambda(\alpha, \beta)$ is a nuclear Fr\'{e}chet space that satisfies $(\Omega)$ and $(\uDN)$.  The fact that $\Lambda(\alpha, \beta)$ is nuclear follows from the general nuclearity criterion for K\"othe sequence spaces (cf.\ \cite[Proposition 28.16]{M-V-IntroFuncAnal}), while it is standard to check that $\Lambda(\alpha, \beta)$ satisfies $(\Omega)$ and $(\uDN)$ (cf.\ \cite[Lemmas 29.11 and 29.12]{M-V-IntroFuncAnal}).
	\end{proof}

We now endow  $\S(\alpha,\beta)$ with the structure of an $m$-convex Fr\'echet algebra, which will enable us to invoke Theorem \ref{l:ClosedIdealHasOmega} in the proof of Theorem \ref{t:main}. Let $f,g : \R_+ \to \C$ be measurable. We define the \emph{Mellin convolution} of $f$ and $g$ as
	\[ f *_{M} g (t) = \int_{0}^{\infty} f(x) g\left(\frac{t}{x}\right) \frac{dx}{x}, \qquad t > 0, \]
provided that these integrals converge.

	\begin{lem}
		\label{t:S(a,b)MConvexAlgebra}
		$(\S(\alpha, \beta), *_{M})$ is an $m$-convex Fr\'{e}chet algebra and the map $\M_{z} : \S(\alpha, \beta) \rightarrow \C$ is a continuous algebra homomorphism for each $z \in \C$ with $\beta < \Re  e \,  z < \alpha$.
	\end{lem}
\begin{proof}
 We define $\S^1(\alpha, \beta)$ as the Fr\'echet space consisting of all $f \in C^{\infty}(\R_{+})$ such that for all $\gamma \in (\beta, \alpha)$ and $n \in \N$ it holds that
	\[ \| f \|^{1}_{\gamma,n} :=  \max_{0 \leq m \leq n}\int^\infty_0 t^{\gamma + m} |f^{(m)}(t)| dt < \infty.  \]
We now claim that $ \S(\alpha, \beta) = \S^1(\alpha, \beta)$ as locally convex  spaces.
Fix $f \in \S(\alpha, \beta)$, $n \in \N$, and $\gamma \in (\beta, \alpha)$. For all $\beta < \gamma_{1} < \gamma < \gamma_{2} < \alpha$ it holds that
	\[ \|f\|^{1}_{\gamma, n} \leq \frac{\|f\|_{\gamma_{1}, n}}{\gamma - \gamma_{1}} + \frac{\|f\|_{\gamma_{2}, n}}{\gamma_{2} - \gamma} . \]
On the other hand, we see that for all $f \in \S^1(\alpha, \beta)$, $n \in \N$, and $\gamma \in (\beta, \alpha)$
	\begin{align*} 
		\|f\|_{\gamma, n} 
		&= \max_{0 \leq m \leq n} |t^{\gamma + m + 1} f^{(m)}(t)| \\
		&= \max_{0 \leq m \leq n} |\int_{0}^{t} (\gamma + m + 1) s^{\gamma + m} f^{(m)}(s) + s^{\gamma + m + 1} f^{(m + 1)}(s) ds| \\
		&\leq (\gamma + n + 1) \|f\|^{1}_{\gamma, n} + \|f\|^{1}_{\gamma, n + 1} .
	\end{align*}
	This shows the claim.
 It holds that
	\[ \| f*_M g \|^1_{\gamma,n} \leq  \| t^\gamma f(t) \|_{L^1}  \| g \|^1_{\gamma,n} \leq \|f\|^1_{\gamma,n} \| g \|^1_{\gamma,n},  \qquad \forall \gamma \in (\beta, \alpha), n \in \N. \]
This implies that $(\S(\alpha, \beta), *_{M})$ is an $m$-convex Fr\'{e}chet algebra. The second statement is clear. 
\end{proof}

Next, we study the space $\S(\alpha, \beta; E)$ with $E$ a locally complete lcHs. Given $x \in \R_+$, we write $\delta_x \in \S(\alpha,\beta)'$ for the Dirac-delta measure concentrated at $x$, i.e.,
$$
\langle \delta_x, \varphi \rangle = \varphi(x), \qquad \varphi \in \S(\alpha,\beta).
$$
Proposition \ref{c:S(a,b)Nuclear} implies that $\S(\alpha,\beta)$ is a Fr\'echet-Schwartz space. Hence,  $\S(\alpha,\beta)'_c =  \S(\alpha,\beta)'_b$, where $b$ stands for the topology of uniform convergence over the bounded sets in $\S(\alpha,\beta)$. We simply denote this space by $\S'(\alpha,\beta)$. Note that $\S'(\alpha, \beta)$ is a $(DFS)$-space. The proof of the next result is standard, we include its proof for the sake of completeness.  
\begin{lem} Let $E$ be a locally complete lcHs.
		\label{p:VVS(alpha,beta)IfEvalInS(alpha,beta)}
		\begin{itemize}
		\item[$(i)$] A function $\mathbf{f} : \R_{+} \rightarrow E$ belongs to $\S(\alpha, \beta; E)$ if and only if $\ev{e'}{\mathbf{f}} \in \S(\alpha, \beta)$  for all $e' \in E'$.
		\item[$(ii)$] The mapping
			\begin{equation}
				\label{eq:IsomorphS(a,b;E)} 
				\Phi: \S(\alpha, \beta) \varepsilon E = L(\S'(\alpha, \beta),E) \to \S(\alpha, \beta; E), \, T \mapsto (x \mapsto T(\delta_x))
			\end{equation}
		is an isomorphism. 
		\item[$(iii)$] For every $z \in \C$ with $ \Re e \, z \in (\beta, \alpha)$ the mapping
				$$
		\M_{E, z}: \S(\alpha, \beta; E) \to E, \quad \M_{E, z}(\mathbf{f}) = \int_{0}^{\infty} t^{z} \mathbf{f}(t) dt
		$$ 
		is well-defined and, under the isomorphism \eqref{eq:IsomorphS(a,b;E)},  it holds that $\M_{E, z} = \M_{z} \varepsilon \id_{E}$. 
\end{itemize}
	\end{lem}
	\begin{proof}
\noindent $(i)$		Let $\mathbf{f} \in \S(\alpha, \beta; E)$. It is clear that $\ev{e'}{\mathbf{f}} \in \S(\alpha, \beta)$ for all $e' \in E^{\prime}$. Conversely, let $\mathbf{f} : \R_{+} \rightarrow E$ be such that $\ev{e'}{\mathbf{f}} \in \S(\alpha, \beta)$ for all $e' \in E^{\prime}$. Lemma \ref{l:WeakCm=>VVCm-1} implies that $\mathbf{f} \in C^{\infty}(\R_{+}; E)$ and $\ev{e'}{\mathbf{f}}^{(m)} = \ev{e'}{\mathbf{f}^{(m)}}$ for all $e' \in E^{\prime}$ and $m \in \N$. Hence, for all $\gamma \in (\beta, \alpha)$ and $n \in \N$  the set
			\[ \{ t^{\gamma + m + 1} \mathbf{f}^{(m)}(t) \mid 0 \leq m \leq n, t >0  \} \]
		is weakly bounded and thus bounded in $E$. This means that $\mathbf{f} \in \S(\alpha, \beta; E)$. \\
\noindent $(ii)$ Note that  $(\S'(\alpha,\beta))' = \S(\alpha,\beta)$. Part $(i)$ implies that $\Phi$ is well-defined. Let $S$ be the linear span of $\{ \delta_x  \mid  x \in \R_+\}$ in $\S'(\alpha,\beta)$. By the Hahn-Banach theorem, $S$ is dense in  $\S'(\alpha,\beta)$. Consequently, $\Phi$ is injective. We now show that it is surjective. Let $\mathbf{f} \in S(\alpha, \beta; E)$ be arbitrary.  We define $T: S \to E$ as the unique linear mapping such that $T(\delta_x) = \mathbf{f}(x)$ for all $x \in \R_+$. Endow $S$ with the subspace topology of $\S'(\alpha,\beta)$.  For all equicontinuous subsets $A \subseteq E'$, it holds that
$$
\sup_{e' \in A} | \langle e', T(g) \rangle | = \sup_{e' \in A} | \langle g, \langle e', \mathbf{f} \rangle \rangle |, \qquad g \in S.
$$
Note that $\{  \langle e', \mathbf{f} \rangle \mid e' \in A \}$ is bounded in $\S(\alpha,\beta)$. Hence, $T: S \to E$ is continuous. Since $\overline{S} = \S'(\alpha,\beta)$ and $\S'(\alpha,\beta)$ is a $(DFS)$-space, \cite[Lemma 6(b)]{BFJ} implies that the local completion \cite[Definition 5.1.21]{BPC} of $S$ is equal to $\S'(\alpha,\beta)$. As $E$ is locally complete, there is  $\widetilde{T} \in L(\S'(\alpha,\beta), E)$ such that $\widetilde{T}_{\mid S} = T$ \cite[Proposition 5.1.25]{BPC}. Clearly,  $\Phi(\widetilde{T}) = \mathbf{f}$.
 \\
\noindent $(iii)$ Let $\mathbf{f} \in \S(\alpha,\beta;E)$. By Lemma \ref{c:WeakC1BornCont}, $\mathbf{f}$  is locally Pettis integrable. Hence, Proposition \ref{p:weak-C0} implies that the Pettis integral  $\M_{E, z}(\mathbf{f})$ exists for each $z \in (\beta, \alpha)$.  The second statement is clear.	
	\end{proof}
	
Finally, we need the following regularization result.
	
	\begin{lem}
		\label{l:VVMellinConvolution}
		Let $E$ be a locally complete lcHs. Let $\mathbf{f} : \R_{+} \rightarrow E$ be a locally Pettis integrable function such that $t^{\gamma} \mathbf{f}(t): \R_+ \to E$ is scalarly integrable for all $\gamma \in (\beta, \alpha)$ and let $\psi \in \S(\alpha, \beta)$. 
		Then, for each $t >0$ the Pettis integral
			\[\mathbf{f} *_{M} \psi(t) = \int_{0}^{\infty} \mathbf{f}(x) \psi\left(\frac{t}{x}\right) \frac{dx}{x} \]
			exists and $\mathbf{f} *_{M} \psi \in \S(\alpha, \beta; E)$. 
	\end{lem}
	
	\begin{proof}
	Proposition \ref{p:weak-C0}  implies that the Pettis integral  $\mathbf{f} *_{M} \psi(t)$ exists for all $t >0$. Let $e' \in E^{\prime}$ be arbitrary. Then, 
	$\ev{e'}{\mathbf{f} *_{M} \psi} \in C^\infty(\R_+)$ and for all  $\gamma \in (\beta, \alpha)$ and $m \in \N$ it holds that
			\begin{align*} 
				&\sup_{t >0}t^{\gamma + m + 1} |\ev{e'}{\mathbf{f} *_{M} \psi(t)}^{(m)}| \\
				&\hspace{2cm} \leq \int_{0}^{\infty} |\ev{e'}{x^{\gamma} \mathbf{f}(x)}| \left(\frac{t}{x}\right)^{\gamma + m + 1} \left|\psi^{(m)}\left(\frac{t}{x}\right)\right| dx \\
				&\hspace{2cm} \leq \|\ev{e'}{x^{\gamma} \mathbf{f}(x)}\|_{L^{1}} \|\psi\|_{\gamma, m}. 
			\end{align*}
		Thus $\ev{e'}{\mathbf{f} *_{M} \psi} \in \S(\alpha, \beta)$. The result now follows from Lemma \ref{p:VVS(alpha,beta)IfEvalInS(alpha,beta)}$(i)$.
	\end{proof}	
	
We are ready to show Theorems \ref{t:main} and \ref{t:MainResultSeqSp}.

	\begin{proof}[Proof of Theorem \ref{t:main}] 
		$(i) \Leftrightarrow (iii)$ Since $E$ is non-empty,  $(iii)$ implies that  $(\M_{z_{n}})_{n \in \N} \subseteq \S^{\prime}(\alpha, \beta)$ is Eidelheit. By Remark \ref{t:ClassicStieltjesMomentProblemEidelheit} the latter is equivalent to the fact that $(z_n)_{n \in \N}$ satisfies $(\condS)$.  Hence, it suffices to show that, under the assumption that  $(\M_{z_{n}})_{n \in \N} \subseteq \S^{\prime}(\alpha, \beta)$ is Eidelheit, $E$ satisfies $(\condA)$ if and only if  $(iii)$ holds.
		 By $(ii)$ and $(iii)$ of Lemma \ref{p:VVS(alpha,beta)IfEvalInS(alpha,beta)},  $(iii)$ is equivalent to the fact that the associated $E$-valued sequence of $(\M_{z_{n}})_{n \in \N}$ is Eidelheit. $\S(\alpha, \beta)$ is a nuclear Fr\'{e}chet space satisfying $(\Omega)$ and $(\uDN)$ (Proposition \ref{c:S(a,b)Nuclear}).	From Theorem \ref{l:ClosedIdealHasOmega} and Lemma \ref{t:S(a,b)MConvexAlgebra} we obtain that
					\[ \{ \varphi \in \S(\alpha, \beta) \mid \M_{z_{n}}(\varphi) = 0,  \forall n \in \N \} \]
		satisfies $(\Omega)$. Hence, the result follows from Theorem \ref{t:Eid}. 
		
	\noindent	$(ii) \Rightarrow (iii)$ Since $E$ is non-empty,  $(ii)$ implies that condition $(ii)$ of Theorem \ref{t:ClassicStieltjesMomentProblem} holds. Hence, this result yields that there is $\psi \in \S(\alpha, \beta)$ such that $\M_{z_{n}}(\psi) = 1$ for all $n \in \N$. Let $(e_{n})_{n \in \N} \in E^{\N}$ be arbitrary. Choose $\mathbf{f} : \R_{+} \rightarrow E$ as in $(ii)$.  Note that $t^{\gamma} \mathbf{f}(t): \R_+ \to E$ is scalarly integrable for all $\gamma \in (\beta, \alpha)$. 
	By Lemma \ref{l:VVMellinConvolution} we have that $\mathbf{f} *_{M}  \psi \in \S(\alpha, \beta; E)$. For all $n \in \N$ and $e' \in E^{\prime}$ it holds that
			\[\langle e', \M_{E, z_{n}}(\mathbf{f} *_{M} \psi) \rangle = \M_{z_{n}}(  \langle e',\mathbf{f} \rangle *_{M} \psi ) =  \M_{z_{n}}(  \langle e',\mathbf{f} \rangle)  \M_{z_{n}}(\psi)= \langle e', e_n \rangle,  \]
		whence $\M_{E, z_{n}}(\mathbf{f} *_{M} \psi ) = e_n$.
		
	\noindent	$(iii) \Rightarrow (ii)$ Obvious as every $\mathbf{f} \in \S(\alpha, \beta;E)$ is locally Pettis integrable (Lemma \ref{c:WeakC1BornCont}).
	\end{proof}

	\begin{proof}[Proof of Theorem \ref{t:MainResultSeqSp}]
	 	We define $E$ as the space consisting of all $(c_{\lambda})_{\lambda \in \Lambda} \in \C^{\Lambda}$ such that  $\sup_{\lambda \in \Lambda}  |c_{\lambda}|  \omega_j(\lambda)< \infty$ for some $j \in \N$. We endow 
		 endow $E$ with its natural  $(LB)$-space topology. Then, $E$ is regular  \cite[Proposition 10]{B-IntroLCIndLim} and thus locally complete. Moreover, $E$ satisfies $(\condA)$ if and only if the condition \eqref{eq:AParamDep} holds. 
		 Since $E$ is regular, a function $\mathbf{f} = (f_\lambda)_{\lambda} :\R_+ \to E$ belongs to $\S(\alpha, \beta;E)$ if and only if $f_\lambda \in \S(\alpha, \beta)$ for all $\lambda \in \Lambda$ and \eqref{charvv} holds. Hence, the equivalence $(i) \Leftrightarrow (iii)$ follows from Theorem \ref{t:main}. The implication $(iii) \Rightarrow (ii)$ is trivial. The fact that  $(ii)$ implies $(iii)$ can be shown in the same way as the implication  $(ii) \Rightarrow (iii)$ in Theorem \ref{t:main}.
	\end{proof}


\end{document}